%% file: main.tex
\documentclass[10pt]{amsart}

\usepackage{fullpage}

\usepackage{mainstyle}

\title[]{Poincare Polynomials of Heavy-Light Hassett Spaces}

\author{Haggai Liu}
\address[H. Liu]{Department of Mathematics\\
Simon Fraser University\\
Burnaby, BC  V5A 1S6 \\
Canada}
\email[H. Liu]{haggail@sfu.ca}

\date{\today}

\begin{document}

    % \nocite{*} % if include all references
    
    \begin{abstract}
        The Poincar\'e polynomials of the Deligne-Mumford space $\overline{M_{0,n}}$ of stable genus 0 curves have been widely studied by several authors such as Keel and Manin. These polynomials can be computed via a recursive formula that is combinatorial in nature, and their exponential generating functions satisfy elegant functional and differential equations.

        In this paper, we state some combinatorial formulas to the Poincar\'e polynomials of Hassett's heavy-light moduli spaces $\overline{M_{0,w_{m,n}}}$, with $m$ heavy marked points and $n-m$ light marked points. We express the Poincar\'e polynomials recursively in terms of the M\"obius function of a certain lattice of set partitions. In the case of $m=2$, we get a Losev-Manin space. We give an explicit formula for the Poincar\'e polynomial for $\overline{M_{0,w_{2,n}}}$ by counting ordered set partitions; and give a recursive formula for this polynomial similar to that in the setting of $\overline{M_{0,n}}$. We prove this recursive formula using geometric and topological properties of the stratification of Losev-Manin spaces; and use an exponential generating function given by Losev and Manin to simplify this recursive formula. Finally, we give a remarkable generalization to the ordered Bell numbers and a recurrence relation for this generalization.
    \end{abstract}

    \maketitle

    \section{Introduction}
        The Deligne-Mumford moduli space $\DMumford{n}$ of stable $n$-pointed genus zero curves has many well understood algebraic, geometric, topological, and combinatorial properties. For example, $\DMumford{n}$ is the wonderful compactification of a projective braid arrangement, and a braid matroid, with respect to their minimal building sets (See Eg. \cite{deconciniP1995, blankersCHLR2024, eurFMPV2026}), as hence are projective varieties. The Poincar\'e polynomial of $\DMumford{n}$ satisfy well-known combinatorial recurrences, which yield elegant functional and differential equations for their exponential generating functions \cite{manin1995,keel1992, getzler1995}. Recently, \cite{aluffiMN2025} expressed the Poincar\'e polynomial of $\DMumford{n}$ explicitly in terms of the Stirling numbers of the first and second kinds. Eur et al. \cite{eurFMPV2026} later reproduced this formula by expressing this polynomial in terms of the reduced characteristic polynomials of braid matroids. For a more detailed treatment of $\DMumford{n}$, we refer the reader to \cite{cavalieri2016}.

        We consider Hassett's \cite{hassett2003} moduli spaces of weighted stable curves $\DMumford{w}$. For such a space, we assign a vector of weights $w=(w_1, \ldots, w_n)\in (\Q\cap (0,1] )^n$ to the marked points. The Hassett space $\DMumford{w}$ parametrizes curves where the marked points are allowed to coincide if the total weight is $\le 1$. We say that the Hassett space $\DMumford{w}$ is \textit{heavy-light} (and denote $w=w_{m,n}$) if $m\in \{2,\ldots, n\}$ of the (heavy) weights are equal to 1, and the other $n-m$ (light) weights are sufficiently close to 0 (see \cref{subsec:heavy-light-hassett} for more details on heavy-light Hassett spaces). In particular, $\DMumford{w_{n,n}}\cong \DMumford{n}$. 

        Poincar\'e polynomials of $\DMumford{w_{m,n}}$ have been studied in various aspects. Chaudhuri \cite{chaudhuri2016} computed the equivariant Poincar\'e polynomials of $\DMumford{w_{m,n}}$ under the natural $S_m\times S_{n-m}$-action permuting the heavy and light marked points respectively amongst each other. Kannan, Serpente, and Yun \cite{kannanSY2024} gives a reduction from results on equivariant Poincar\'e polynomials to results on ordinary Poincar\'e polynomials in arbitrary genus. This reduction involves a certain rank function from symmetric functions to $\Q$-power series.
        
        Losev-Manin spaces \cite{losevM2000} are heavy-light Hassett spaces with exactly two heavy marked points, so they parametrize ``caterpillar shaped'' curves. Bergstr\"om and Minabe \cite{bergstromM2011} computed the $S_2\times S_{n-2}$-equivariant Poincar\'e polynomial of $\LM{n}$.
        
        The goal of this paper is to obtain elegant recurrences for the Poincar\'e polynomials of $\DMumford{w_{m,n}}$, and give combinatorial proofs for those recurrences by examining the standard stratifications of these Hassett spaces. The methods used in \cite{bergstromM2011,chaudhuri2016,kannanSY2024} directly studies the cohomology of the moduli spaces with extensive use of representation theory, expressing the Poincar\'e polynomials in terms of Schur functions. We using a purely combinatorial approach and obtain recursive formulas by studying the stratifications of these moduli spaces by their degeneration type.

        \subsection{Statement of Results} We summarize the main results in this paper concerning the Poincar\'e polynomials $p_{\DMumford{w_{m,n}}}(x)$ of heavy-light Hassett spaces. In what follows, we write $S(n,k)$ to denote a Stirling number of the second kind, and let $\mu$ be the M\"obius function of a set partition lattice $\Pi_F$ (See \cref{subsec:set-partitions} for precise statements). Our first main result implies that $p_{\DMumford{w_{m,n}}}(x)$ can be computed recursively using the Poincar\'e polynomials of heavy-light Hassett spaces of smaller dimensions.

        \begin{thm}[\cref{thm:poincare-poly-HL-hassett-mu-rec}]
            Let $F:=[n]\setminus[m]=\{m+1,\ldots, n\}$.
            The Poincar\'e polynomial of $\DMumford{w_{m,n}}$ satisfies the recursive formula 
            $$p_{\DMumford{w_{m,n}}}(x)=p_{\DMumford{n}}(x)-\sum_{\mcP\in \Pi_F} \left[ \left(\sum_{\mcQ\in [\hat{0}_F, \mcP]\setminus \{\hat{1}_F\}} \mu(\mcQ, \mcP) p_{\DMumford{w_{m,m+|\mcQ|}}}(x) \right)\left(\prod_{A\in \mcP}p_{\DMumford{|A|+1}}(x)\right) \right].$$
        \end{thm}

        In the special case of Losev-Manin spaces where $m=2$, we have the explicit formula below.

        \begin{thm}[\cref{thm:poincare-poly-losev-manin-space}]
            Let $(x)_k:={x\choose k} k!$ denote the falling factorial polynomial. Then the Poincar\'e polynomial of the Losev-Manin space $\DMumford{w_{2,n}}$ satisfies the formula 
            $$p_{\DMumford{w_{2,n}}}(x)=\sum_{(l_1, \ldots, l_k)}S(n-2, \sum l_i){\sum l_i\choose l_1,l_2,\ldots, l_k}\prod_{i=1}^k (x-2)_{l_i-1}$$
            where $(l_1, \ldots, l_k)$ range over sequences of positive integers with $1\le \sum l_i\le n-2$.
        \end{thm}

        Additionally, we also have the following results concerning $p_{\LM{n}}(x)$.  

        \begin{thm}
            Let $p_n(x):=p_{\LM{n+2}}(x)$, with $p_0(x)=1$.
            \begin{enumerate}[itemsep=10pt]
                \item (\cref{thm:LM-ppoly-deg2-recursion}) $p_{n-1}(x)=(1+x)p_{n-2}(x)+x\sum_{j=1}^{n-3} {n-2\choose j} p_j(x)p_{n-j-2}(x)$.
                \item (\cref{cor:egf-ppoly-LM-space}) The exponential generating function $P(x,t):=\sum_{n\ge0}p_n(x)\frac{t^n}{n!} = \frac{1-x}{e^{(x-1)t}-x}$ is a solution to the differential equation $P_t=xP^2+(1-x)P$.
                \item (\cref{cor:deg1rec-ppoly-LM-space}) $p_n(x)=\sum_{i=1}^n {n\choose i} (x-1)^{i-1}p_{n-i}(x)$.
            \end{enumerate}
        \end{thm}

        Note that Losev and Manin have proven the same formula for the exponential generating function of the polynomials $p_{\LM{n}}(x)$ in \cite[Theorem 2.3]{losevM2000} by counting moduli points of the finite field $\mathbb{F}_q$. We prove this same formula by first proving a recurrence relation by examining the degeneration type of stable curves.

        Finally, we use these results on $p_{\LM{n}}(x)$ to generalize the order Bell numbers and to prove a recursive formula for these generalizations.
        
    \section{Background}
        \subsection{Boundary Strata in $\DMumford{n}$}
            Let $T$ be a tree with leaves labelled by numbers $1,\ldots, n$ where each internal vertex has degree $\ge3$. We refer to such a tree as a \textit{stable tree}. We say that $T$ \textit{refines} another stable tree $T'$ if $T'$ is obtained from $T$ by contracting some (possibly none) of its internal edges. Each stable tree $T$ is associated to a locally closed subset $X_T\subseteq\DMumford{n}$ consisting of the pointed curves in $\DMumford{n}$ with dual tree $T$.
            We refer to $X_T$ as an open boundary stratum and its closure $\overline{X_T}$ as a {closed boundary stratum}.
            For a stable trees $T$ and $T'$, we have that $X_{T'}\subseteq \overline{X_T}$ if and only if $T'$ is a refinement of $T$. Hence,
            $$\overline{X_T}=\coprod_{T'}X_{T'}=\bigcup_{T'}\overline{X_{T'}}$$
            where $T'$ ranges over the refinements of $T$.
    
            % Finally, for each subset $I\subseteq [n]$ with $2\le |I|\le n-2$, we denote $I^c:=[n]-I$ and 
            % $D_I=D(I|I^c):=X_T$, where $T$ is the stable tree with exactly one internal edge $uv$ such that the leaves indexed by $I$ are adjacent to $u$, and the other $n-|I|$ leaves are adjacent to $v$. In general, for a stable tree $T$ and an internal edge $e$ of $T$, we define $D_e:=D(A|B)$, where the deletion of $e$ from $T$ forms two connected components, one of which contains exactly the leaves labelled by indices in $A$; and the other contains exactly the leaves labelled by indices in $B$.
    
            % \begin{exmp}
            %     Example here to demonstrate the $D_e$ notation
            % \end{exmp}

        \subsection{Poincar\'e Polynomials} 
            Associated to each variety $X$ over $\C$, is a particular polynomial $p_X(x)\in \Z[x]$ known as the \textit{virtual Poincar\'e Polynomial} of $X$. Such polynomials are characterized uniquely by the following properties \cite{fultonM1994}:
            \begin{enumerate}
                \item If $X$ is smooth and compact (as a $\C$-manifold), then 
                $$p_X(x) = \sum_{i=0}^{\dim(X)} \dim A^i(X) x^i$$
                is the ordinary \textit{Poincar\'e polynomial} of $X$. Observe that, in this case, $p_X(x)$ is palindromic due to Poincar\'e duality.
                \item if $X=\coprod X_i$ is a decomposition into finitely many locally closed subsets $X_i$, then $p_X(x)=\sum p_{X_i}(x)$. 
            \end{enumerate}

            It also follows that Poincar\'e polynomials are multiplicative in the sense that $p_{X\times Y}(x) = p_{X}(x)p_Y(x)$.
    
            \begin{exmp}
                Consider the moduli space $\DMumford{n}$. Its Poincar\'e polynomial satisfies the equation 
                $$p_{\DMumford{n}}(x)=\sum_T p_{X_T^\circ}(x)$$
                where $T$ ranges over all stable trees with leaves labelled $1,\ldots, n$. It is known that $p_{\moduli{n}}(x)={x-2\choose n-3}(n-3)!$. See eg. \cite[Prop 1.1.2]{manin1995}.
            \end{exmp}
    
            There has been numerous works on the Poincar\'e polynomial of the Deligne-Mumford moduli space $\DMumford{n}$. See eg. \cite{keel1992,manin1995}. Very recently, \cite{eurFMPV2026} developed new combinatorial formulas and reproduced known formulas for $p_{\DMumford{n}}(x)$ by examining polymatroids and their Hilbert series over a building set. For example, one such formula is 
            \begin{equation}
            \label{eqn:poincare-DMspace-recursive}
                p_{\DMumford{n+1}}(x) = p_{\DMumford{n}}(x) + x\sum_{j=2}^{n-1} {n-1\choose j} p_{\DMumford{j+1}}(x)p_{\DMumford{n-j+1}}(x)
            \end{equation}
            where $p_{\DMumford{n}}(x)=1$ if $n=2,3$ \footnote{We consider $\DMumford{2}$ as a point for notational convenience}. \cref{tab:poincare-poly-Mon-bar} shows the values of $p_{\DMumford{n}}(x)$ up to $n=12$. 
            
            \begin{rmk}
                The coefficient of the linear term of $p_{\DMumford{n}}(x)$ is the rank of the divisor class group of the moduli space $\DMumford{n}$ and can be computed by the formula 
                $$p'_{\DMumford{n}}(0)=2^{n-1}-1-{n\choose 2}.$$
                See \cite[p. 550]{keel1992}.
            \end{rmk}
            
            The recursive formula \eqref{eqn:poincare-DMspace-recursive} is already established in \cite[Cor. 0.3.2]{manin1995}, but \cite{eurFMPV2026} provides an alternative proof using FY-monomials, which are combinatorial and are direct properties of polymatroids and their building sets.

            \begin{table}[]
                \centering
                \begin{tabular}{c|c}
                   $n$  & $p_{\DMumford{n}}(x)$ \\
                   \hline
                    $3$  &  $1$\\
                    $4$ & $x+1$\\
                    $5$ & ${x}^{2}+5\,x+1$\\
                    $6$ & ${x}^{3}+16\,{x}^{2}+16\,x+1$\\
                    $7$ & ${x}^{4}+42\,{x}^{3}+127\,{x}^{2}+42\,x+1$\\
                    $8$ & ${x}^{5}+99\,{x}^{4}+715\,{x}^{3}+715\,{x}^{2}+99\,x+1$\\
                    $9$ & ${x}^{6}+219\,{x}^{5}+3292\,{x}^{4}+7723\,{x}^{3}+3292\,{x}^{2}+219\,x+1$\\
                    $10$ & ${x}^{7}+466\,{x}^{6}+13333\,{x}^{5}+63173\,{x}^{4}+63173\,{x}^{3}+13333\,{x}^{2}+466\,x+1$\\
                    $11$ & ${x}^{8}+968\,{x}^{7}+49556\,{x}^{6}+429594\,{x}^{5}+861235\,{x}^{4}+429594\,{x}^{3}+49556\,{x}^{2}+968\,x+1$\\
                    $12$ & ${x}^{9}+1981\,{x}^{8}+173570\,{x}^{7}+2567940\,{x}^{6}+9300303\,{x}^{5}+9300303\,{x}^{4}+2567940\,{x}^{3}+173570\,{x}^{2}+1981\,x+1$
                \end{tabular}
                \caption{Poincar\'e polynomials of $\DMumford{n}$ for $3\le n\le 12$.}
                \label{tab:poincare-poly-Mon-bar}
            \end{table}

    \subsection{Heavy-Light Hassett Spaces}
    \label{subsec:heavy-light-hassett}

        Hassett \cite{hassett2003} introduced weighted analogues $\DMumford{w}$ of $\DMumford{n}$. For a rational weight vector $w=(w_1, \ldots, w_n) \in (\Q\cap (0,1])^n$ with $\sum w_i>2$, a \textit{$w$-stable curve} is a rational $n$-pointed curve such that marked points that coincide have total weight $\le 1$. We use $\epsilon$ to denote a tiny positive weight. That is, marked points of weight $\epsilon$ can coincide with other marked points with total weight $<1$. The Hassett space $\DMumford{w}$ is the moduli space of $w$-stable curves up to equivalence. Following the terminology in \cite{kannanKL2021}, we refer to a weight vector of the form
        $$w=w_{m,n}:=(1^m, \epsilon^{n-m})$$
        as \textit{heavy-light}. We say that a Hassett space is heavy-light if it is isomorphic to a Hassett space of the form $\DMumford{w_{m,n}}$. In the special case when $m=2$, such a heavy-light Hassett space is a Losev-Manin space \cite{losevM2000}.

        \begin{rmk}\ 
            \begin{itemize}
                \item There is a reduction map $\rho:\DMumford{n}\onto \DMumford{w_{m,n}}$ that successively collapses components that become unstable. If $m\ge n-2$, then $\rho$ is an isomorphism.

                \item More generally, if $r\le m\le n$, then there is a reduction map $\DMumford{w_{m,n}}\onto \DMumford{w_{r,n}}$.
            \end{itemize}
        \end{rmk}
    
    \subsection{Chow ring of $\DMumford{n}$ and $\DMumford{w_{m,n}}$}
        The presentation of the Chow ring $A(\DMumford{n})$ was first computed by Keel, dating back to 1992, using a description of $\DMumford{n}$ as an iterated blowup of $(\Pj^1)^{n-3}$. 
        \begin{thm}[\cite{keel1992}]
            Let $R=\Z[x_I: 2\le |I| \le n-2]$ be the polynomial ring over $\Z$ on $2^n-2n-2$ variables. Then there is a quotient map $R\onto A(\DMumford{n})$ given by $x_I\mapsto [D_I]$. Furthermore, the kernel of this quotient map is the ideal generated by the following elements:
            \begin{enumerate}
                \item $x_I-x_{I^c}$ for all $I$;
                \item $x_I x_J$ for each pair $I,J$ such that $\varnothing\notin \{I\cap J, I \cap J^c, I^c\cap J, I^c \cap J^c\};$
                \item (WDVV relations) for $\{a,b,c,d\}\in {[n]\choose 4}$, 
                $$\sum_{I: I\cap\{a,b,c,d\}=\{a,b\}} x_I - \sum_{I: I\cap\{a,b,c,d\}=\{a,c\}} x_I.$$
            \end{enumerate}
        \end{thm}

        More recently, Kannan, Karp, and Li~\cite{kannanKL2021} used tools from tropical geometry to extend the above result to a presentation of the Chow ring of heavy-light Hassett spaces.

        \begin{thm}[\cite{kannanKL2021}]
            Fix $m\ge2$ and $n\ge 4$ and consider the weight vector $w=w_{m,n}:=(1^m, \epsilon^{n-m})$. 
            Let 
            $$R=\Z[x_I: I\subsetneq \{2,\ldots, n\};\ \{2,\ldots, m\}\cap I \neq \varnothing; \ |I|>1]$$ 
            be the polynomial ring over $\Z$ on $2^{n-1}-2^{n-m}-m$ variables. Then there is a quotient map $R\onto A(\DMumford{w})$. Furthermore, the kernel of this quotient map is the ideal generated by the following elements:
            \begin{enumerate}
                \item $x_I x_J$ for each pair $I,J$ such that $I\cap J\notin \{I,J, \varnothing\} $;
                \item For each $\{i,j\},\{k,l\}\in {\{2,\ldots, n\}\choose 2}$ with $i,k\le m$, 
                $$\sum_{I: \{i,j\}\subseteq I, \{k,l\}\not\subseteq I} x_I - \sum_{I: \{k,l\}\subseteq I, \{i,j\}\not\subseteq I} x_I.$$
            \end{enumerate}
        \end{thm}

    \section{Stratifications of $\DMumford{w_{m,n}}$}

        \subsection{Set partitions}
        \label{subsec:set-partitions}
            Fix a non-empty finite set $F$ consisting of $n$ elements. We denote $S(n,k)$ as the Stirling number of the second kind that counts the number of partitions of $F$ into exactly $k$ parts. Let $B(n):=\sum_{k=1}^n S(n,k)$ denote the $n$-th Bell number. If $\mcP$ and $\mcQ$ are partitions of $F$, we say that $\mcP$ \textit{refines} $\mcQ$ are write $\mcQ\le \mcP$ if each set in $\mcQ$ is a union of sets in $\mcP$. If $\mcQ\le \mcP$, with $\mcQ=\{Q_1, \ldots, Q_k\}$, and $\mcP=\cup_{i=1}^k \{P_{i1},\ldots. P_{il_i}\}$, where each $Q_i=\cup_j P_{ij}$, then we define 
            $$\mu(\mcQ, \mcP):=\prod_{i=1}^k (-1)^{l_i-1} (l_i-1)!.$$
            The function $\mu$ is the M\"obius function of of the lattice $\Pi_F$.
            Furthermore, if we fix an order $Q_1, \ldots, Q_k$ on the parts of $\mcQ$, then we say that the pair $\mcQ\le \mcP$ is of \textit{type $(l_1, \ldots, l_k)$}.
            The partitions of $F$ form a lattice $\Pi_F$ under refinement. Consider the free abelian group 
            $$\Z \Pi_F := \Z\cdot \{e_\mcP : \mcP\in \Pi_F\}$$ 
            with basis elements indexed by the partitions of $F$. For each $\mcP\in \Pi_F$, define 
            $$\overline{e}_{\mcP}:=\sum_{\mcQ\le \mcP} e_{\mcQ}$$
            to be the sum of all basis elements in the lower order ideal of $\mcP$. It follows that the set $\{\overline{e}_\mcP : \mcP\in \Pi_F\}$ is also a $\Z$-basis for $\Z\Pi_F$ and 
            $${e}_{\mcP}=\sum_{\mcQ\le \mcP} \mu(\mcQ,\mcP)\overline{e}_{\mcQ}.$$

            \begin{rmk}
                The partition lattice $\Pi_F$ has a unique minimal element $\hat{0}_F:= \{F\}$; and a unique maximal element $\hat{1}_F:=\{\{x\}: x\in F\}$. 
            \end{rmk}

        \subsection{$w$-stable trees} 
            Let $w=(w_1, \ldots, w_n)$ be a weight vector. A $w$-stable tree is a tree $T$ where each internal vertex has degree $\ge 3$, and the leaves of $T$ are labelled by subsets of $[n]$ such that the leaf labels from a partition $\mcP=\mcP(T)$ of $[n]$ with the condition that for each $A\in \mcP$, $\sum_{i\in A} w_i\le 1$. Similar to the unweighted case, each $w$-weighted stable tree indexes a locally closed subset $X_T$ consisting of the $w$-stable curves with dual tree $T$.

        \subsection{Subsets of $\DMumford{w_{m,n}}$ indexed by set partitions}
            Consider the heavy-light Hassett space $\DMumford{w_{m,n}}$ and set $F:=[n]\setminus [m]=\{m+1, m+2, \ldots, n\}$. For each partition $\mcP\in \Pi_F$, define the locally closed subset 
            $$X_{\mcP}:=\coprod_{T}X_T\subseteq \DMumford{w_{m,n}}$$
            where $T$ ranges from the set of $w_{m,n}$-stable trees such that $\mcP(T)=\mcP\cup \hat{1}_{[m]}$. It follows that $\overline{X_{\mcP}} = \coprod_{\mcQ\le \mcP} X_{\mcP}$, which yields to following equality of Poincar\'e polynomials:
            $$p_{\overline{X_{\mcP}}}(x)=\sum_{\mcQ\le \mcP}p_{X_{\mcP}}(x).$$
            In particular, $\overline{X_{\hat{1}_F}}=\DMumford{w_{m,n}}$ is the entire Hassett space. Additionally, a projection map $\DMumford{w_{m,n}}\onto \DMumford{w_{m,|\mcP|}}$ defined by forgetting certain light marked points, restricts to an isomorphism $\overline{X_{\mcP}}\cong \DMumford{w_{m, m+|\mcP|}}$. That is, the closures of the subsets $X_{\mcP}$ indexed by $\mcP\in \Pi_F$ are themselves heavy-light Hassett spaces of smaller dimension. We state below, an explicit formula for the Poincar\'e polynomial of Losev-Manin spaces. 

            \begin{thm}
            \label{thm:poincare-poly-losev-manin-space}
                Let $(x)_k:={x\choose k} k!$ denote the falling factorial polynomial. Then the Poincar\'e polynomial of the Losev-Manin space $\DMumford{w_{2,n}}$ satisfies the formula 
                $$p_{\DMumford{w_{2,n}}}(x)=\sum_{(l_1, \ldots, l_k)}S(n-2, \sum l_i){\sum l_i\choose l_1,l_2,\ldots, l_k}\prod_{i=1}^k (x-2)_{l_i-1}$$
                where $(l_1, \ldots, l_k)$ range over sequences of positive integers with $1\le \sum l_i\le n-2$.
            \end{thm}

            \begin{proof}
                Let $w=w_{2,n}$. The $w$-stable trees are exactly of the form \\
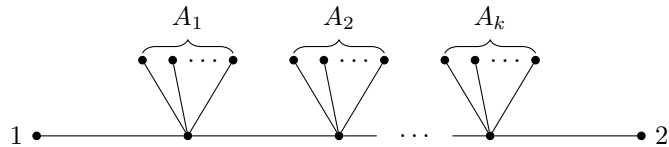
                $$
\input{w2n-LM-caterpillar-tree}
                $$\\
                where $(A_1, \ldots, A_k)$ is a composition of $F:=\{3,\ldots, n\}$. A $w$-stable trees of the above form is uniquely characterized by a refinement of the partition $\{A_1, \ldots, A_k\}$. Thus, the strata of $\DMumford{w}$ are in bijection with length two chains $\mcQ\le \mcP$ in $\Pi_F$ together with an order on the parts of $\mcQ$. Such a pair of type $(l_1, \ldots, l_k)$ correspond to a stratum combinatorially equivalent to $\prod_i \moduli{l_i+2}$. There are exactly 
                $$S(n-2, \sum l_i){\sum l_i\choose l_1,l_2,\ldots, l_k}$$
                pairs of type $(l_1, \ldots, l_k)$. The result now follows from the additive and multiplicative properties of virtual Poincar\'e polynomials.
            \end{proof}

            \begin{exmp}[$m=2,n=5$]
            \label{ex:comp-losev-manin-2D}
                We demonstrate how the formula in \cref{thm:poincare-poly-losev-manin-space} can be used to compute the Poincar\'e polynomial of the two-dimensional Losev-Manin space $\DMumford{w_{2,5}}$, the blow-up of $\Pj^2$ at four points. Using the data summarized in \cref{tab:comp-data-losev-manin}, it follows that 
                $$p_{\DMumford{w_{2,5}}}(x) = 1+6+6+(3+3+3)(x-2)+(x-2)(x-3)=x^2+4x+1.$$
                
                \renewcommand{\arraystretch}{2}
                \begin{table}[]
                    \centering
                    \begin{tabular}{|c|c|c|c|c|}
                        \hline
                         $(l_1, \ldots, l_k)$ & $\sum l_i$ & $S(3, \sum l_i)$ & ${\sum l_i\choose l_1,l_2,\ldots, l_k}$ & $\prod_{i=1}^k (x-2)_{l_i-1}$\\
                         \hline
                         $(1)$& $1$ & $1$ & $1$ & $1$\\
                         $(1,1)$& $2$ & $3$ & $2$ & $1$ \\
                         $(2)$& $2$ & $3$ & $1$ & $x-2$ \\
                         $(1,1,1)$& $3$ & $1$ & $6$ & $1$ \\
                         $(1,2)$& $3$ & $1$ & $3$ & $x-2$ \\
                         $(2,1)$& $3$ & $1$ & $3$ & $x-2$ \\
                         $(3)$& $3$ & $1$ & $1$ & $(x-2)(x-3)$ \\
                         \hline
                    \end{tabular}
                    \caption{Data for the computation in \cref{ex:comp-losev-manin-2D}.}
                    \label{tab:comp-data-losev-manin}
                \end{table}
            \end{exmp}
            
            Define the ring map $\psi:\Z\Pi_F\to \Z[x]$ given by $\psi({e_{\mcP}}):=p_{X_{\mcP}}(x)$. Then 
            $$\psi(\overline{e}_\mcP)=p_{\overline{X_{\mcP}}}(x)=p_{\DMumford{w_{m,m+|\mcP|}}}(x).$$

            The reduction map $\DMumford{n}\onto \DMumford{w_{m,n}}$ is a Zariski trivial fibration over $X_{\mcP}$, with fiber isomorphic to the product of Deligne-Mumford spaces 
            $$\prod_{A\in \mcP} \DMumford{|A|+1}.$$
            It follows that 
            $$\DMumford{n}\cong \coprod_{\mcP\in \Pi_F} Y_{\mcP}$$
            where each $Y_{\mcP}\subseteq \DMumford{n}$ is locally closed and 
            $$Y_{\mcP}\cong \left(X_{\mcP}\times \left(\prod_{A\in \mcP}\DMumford{|A|+1}\right)\right).$$

            The Poincar\'e polynomials therefore satisfy 
            \begin{align*}
                p_{\DMumford{n}}(x)&=\sum_{\mcP\in \Pi_F} p_{X_{\mcP}}(x)\left(\prod_{A\in \mcP}p_{\DMumford{|A|+1}}(x)\right)\\
                &=\sum_{\mcP\in \Pi_F} \left(\sum_{\mcQ\in [\hat{0}_F, \mcP]} \mu(\mcQ, \mcP) p_{\overline{X_{\mcQ}}}(x) \right)\left(\prod_{A\in \mcP}p_{\DMumford{|A|+1}}(x)\right)\\
                &=\sum_{\mcP\in \Pi_F} \left(\sum_{\mcQ\in [\hat{0}_F, \mcP]} \mu(\mcQ, \mcP) p_{\DMumford{w_{m,m+|\mcQ|}}}(x) \right)\left(\prod_{A\in \mcP}p_{\DMumford{|A|+1}}(x)\right).
            \end{align*}

            This allows us express the Poincar\'e polynomial of $\DMumford{w_{m,n}}$ in terms of $\DMumford{n}$, and heavy-light Hassett spaces of smaller dimension.

            \begin{thm}\label{thm:poincare-poly-HL-hassett-mu-rec}
                Let $F:=[n]\setminus[m]=\{m+1,\ldots, n\}$.
                The Poincar\'e polynomial of $\DMumford{w_{m,n}}$ satisfies the recursive formula 
                $$p_{\DMumford{w_{m,n}}}(x)=p_{\DMumford{n}}(x)-\sum_{\mcP\in \Pi_F} \left[ \left(\sum_{\mcQ\in [\hat{0}_F, \mcP]\setminus \{\hat{1}_F\}} \mu(\mcQ, \mcP) p_{\DMumford{w_{m,m+|\mcQ|}}}(x) \right)\left(\prod_{A\in \mcP}p_{\DMumford{|A|+1}}(x)\right) \right].$$
            \end{thm}

            \begin{cor}
            \label{cor:poincare-meq-nminus3-formula}
                If $m=n-3$, then 
                $$p_{\DMumford{w_{m,n}}}(x)=p_{\DMumford{n}}(x)-xp_{\DMumford{n-2}}(x).$$
            \end{cor}

            \begin{proof}
                Let $F:=\{n-2,n-1,n\}$. For $\mcQ\le \mcP\in \Pi_F$, we have 
                $$\mu(\mcQ, \mcP) = 
                \begin{cases}
                    1, & \mcQ=\mcP;\\
                    -1,  & \mcQ = \hat{0}_F, |\mcP|=2;\\
                    2, & \mcQ = \hat{0}_F, \mcP=\hat{1}_F;\\
                    -1,  & |\mcQ| = 2, \mcP=\hat{1}_F.
                \end{cases}$$
                A direct application of the equation in \cref{thm:poincare-poly-HL-hassett-mu-rec} yields
                \begin{align*}
                    p_{\DMumford{n}}(x)-p_{\DMumford{w_{m,n}}}(x) &= 3p_{\DMumford{n-1}}(x)((1)(1)+(-1)(1))+p_{\DMumford{n-2}}(x)((1)(1+x)+3(-1)(1)+(2)(1))\\
                    &= xp_{\DMumford{n-2}}(x).\qedhere
                \end{align*}
            \end{proof}

            \begin{rmk}
                Since $p_{\DMumford{5}}(x)=x^2+5x+1$, \cref{cor:poincare-meq-nminus3-formula} recovers the formula 
                $$p_{\DMumford{w_{2,5}}}(x)=x^2+4x+1$$
                that was computed in \cref{ex:comp-losev-manin-2D}.
            \end{rmk}

        \section{Recursive Formulas for Losev-Manin Spaces}

            In this, section, we fix the number of heavy marked points to be $m=2$ and strictly consider the Losev-Manin space $\LM{n}$. Their poincar\'e polynomials satisfy an elegant recursive formula, which was use here to reproduce the exponential generating function in \cite[Theorem 2.3]{losevM2000}.

            \begin{thm}\label{thm:LM-ppoly-deg2-recursion}
                Denote $p_n(x):=p_{\LM{n+2}}(x)$ for $n\ge 1$, and $p_0(x):=1$. Then for $n\ge 3$, we have 
                $$p_{n-1}(x)=(1+x)p_{n-2}(x)+x\sum_{j=1}^{n-3} {n-2\choose j} p_j(x)p_{n-j-2}(x).$$
            \end{thm}

            \begin{proof}

\input{proof-of-thm-LM-ppoly-deg2-recursion}
            \end{proof}

            \begin{cor}
            \label{cor:egf-ppoly-LM-space}
                Let $p_n(x)$ be as defined in \cref{thm:LM-ppoly-deg2-recursion}; and let
                $$P(x,t)=\sum_{n\ge 0} p_n(x)\frac{t^n}{n!}\in \Q(x)[[t]]$$
                be their exponential generating function. Then $P(x,t)$ satisfies the differential equation 
                \begin{equation}
                \label{eqn:diffeq-egf-ppoly-LM}
                    P_t=xP^2+(1-x)P
                \end{equation}
                and hence,
                \begin{equation}
                \label{eqn:egf-ppoly-LM}
                    P(x,t)=\frac{1-x}{e^{(x-1)t}-x},
                \end{equation}
                which agrees with \cite[Theorem 2.3]{losevM2000}.
            \end{cor}

            \begin{proof}
                The differential equation \eqref{eqn:diffeq-egf-ppoly-LM} follows from straightforward algebraic manipulation of $P(x,t)$ using the recurrence in \cref{thm:LM-ppoly-deg2-recursion}. Since $P(x,0)\equiv 1$, Eqn \eqref{eqn:egf-ppoly-LM} holds.
            \end{proof}

            \begin{cor}
            \label{cor:deg1rec-ppoly-LM-space}
                Let $p_n(x)$ be as defined in \cref{thm:LM-ppoly-deg2-recursion}. Then for $n\ge 1$, 
                \begin{equation}
                \label{eqn:LM-ppoly-linear-rec}
                    p_n(x)=\sum_{i=1}^n {n\choose i} (x-1)^{i-1}p_{n-i}(x).
                \end{equation}
            \end{cor}

            \begin{proof}
                Consider the series expansion at $t=0$:
                $$\left(e^{(x-1)t}-x\right)^{-1}=\left((1-x)+\sum_{n\ge 1}\frac{1}{n!}(x-1)^n t^n\right)^{-1} = \sum_{n\ge 0} f_n(x) t^n \in \Q(x)[[t]].$$
                The sequence $n\mapsto f_n(x)$ satisfy the recurrence 
                $$f_n(x)=\frac{-1}{1-x} \sum_{i=1}^n \frac{1}{i!}(x-1)^i f_{n-i}(x)$$
                with $f_0(x)=\frac{1}{1-x}$. Eqn \eqref{eqn:egf-ppoly-LM} implies that $p_n(x)=f_n(x)n!(1-x)$. The desired recurrence for $p_n(x)$ is now immediate.
            \end{proof}

            \begin{table}[]
                \centering
                \begin{tabular}{c|c}
                   $n$  & $p_{\LM{n}}(x)$ \\
                   \hline
                    $3$ & $1$\\
                    $4$ & $x+1$\\
                    $5$ & ${x}^{2}+4\,x+1$\\
                    $6$ & ${x}^{3}+11\,{x}^{2}+11\,x+1$\\
                    $7$ & ${x}^{4}+26\,{x}^{3}+66\,{x}^{2}+26\,x+1$\\
                    $8$ & ${x}^{5}+57\,{x}^{4}+302\,{x}^{3}+302\,{x}^{2}+57\,x+1$\\
                    $9$ & ${x}^{6}+120\,{x}^{5}+1191\,{x}^{4}+2416\,{x}^{3}+1191\,{x}^{2}+120\,x+1$\\
                    $10$ & ${x}^{7}+247\,{x}^{6}+4293\,{x}^{5}+15619\,{x}^{4}+15619\,{x}^{3}+4293\,{x}^{2}+247\,x+1$\\
                    $11$ & ${x}^{8}+502\,{x}^{7}+14608\,{x}^{6}+88234\,{x}^{5}+156190\,{x}^{4}+88234\,{x}^{3}+14608\,{x}^{2}+502\,x+1$\\
                    $12$ & ${x}^{9}+1013\,{x}^{8}+47840\,{x}^{7}+455192\,{x}^{6}+1310354\,{x}^{5}+1310354\,{x}^{4}+455192\,{x}^{3}+47840\,{x}^{2}+1013\,x+1$

                \end{tabular}
                \caption{Poincar\'e polynomials of $\LM{n}$ for $3\le n\le 12$.}
                \label{tab:poincare-poly-LM-space}
            \end{table}

            \begin{exmp}[Losev-Manin space with $6$ marked points]
                We have 
                \begin{align*}
                    p_{\LM{6}}(x)&=p_4(x)\\
                    &= {4\choose 1} p_3(x) + {4\choose 2} (x-1) p_2(x) + {4\choose 3} (x-1)^2 p_1(x) + {4\choose 4} (x-1)^3 p_0(x)\\
                    &= {4\choose 1} (x^2+4x+1) + {4\choose 2} (x-1)(x+1) + {4\choose 3} (x-1)^2(1) + {4\choose 4} (x-1)^3 (1)\\
                    &= x^3+11x^2+11x+1.
                \end{align*}
                Alternatively, using the recurrence given in \cref{thm:LM-ppoly-deg2-recursion}, 
                \begin{align*}
                    p_4(x)&=(1+x)p_3(x)+x\left({3\choose 1} p_1(x)p_2(x)+{3\choose 2}p_2(x)p_1(x)\right)\\
                    &= (x+1)(x^2+4x+1)+6x(x+1)\\
                    &= x^3+11x^2+11x+1.
                \end{align*}
            \end{exmp}

            \noindent \cref{tab:poincare-poly-LM-space} shows the values of $p_{\LM{n}}(x)$ up to $n=12$.

            \begin{rmk}
                It follows from either \cref{thm:LM-ppoly-deg2-recursion} or Eqn \eqref{eqn:LM-ppoly-linear-rec} that the coefficient of the linear term of $p_{\LM{n}}(x)$ is the rank of the divisor class group of the Losev-Manin moduli space $\LM{n}$ and can be computed by the formula 
                $$p'_{\LM{n}}(0)=2^{n-2}-n+1.$$
            \end{rmk}

            \begin{rmk}
                At $x=2$, the equation in \cref{thm:poincare-poly-losev-manin-space} yields
                $$p_{\LM{n}}(2) = \sum_{k=1}^{n-2} S(n-2, k) k!$$
                which is the $(n-2)$-th ordered Bell number. Substituting $x=2$ into Eqn \eqref{eqn:LM-ppoly-linear-rec} yields a known recurrence for the ordered Bell numbers.
            \end{rmk}

            The above results give a remarkable generalization of the ordered Bell numbers. For an integer $r\ge 1$, we define the \textit{ordered $r$-Bell numbers} as 
            $$b(n,r):=\sum_{N_1, \ldots, N_r \in \Zge;\ \sum N_j>0} S(n, \sum jN_j) (\sum jN_j)! {\sum N_j \choose N_1,\ldots, N_r} \left(\frac{1}{r}\right)^{\sum N_j} \prod_j {r\choose j}^{N_j}.$$
            Observe that $b(n,r)=p_{\LM{n+2}}(r+1)$, where the number $N_j$ represents the number of occurrences of $j$ in the sequence $(l_1, \ldots, l_k)$, in the formula given by \cref{thm:poincare-poly-losev-manin-space}. From this relation between ordered $r$-Bell numbers and Poincar\'e polynomials, we immediately have the following result.

            \begin{thm}
                The ordered $r$-Bell numbers $b(n,r)$ defined above satisfies the recurrence 
                $$b(n,r) = r^{n-1}+\sum_{i=1}^{n-1} {n\choose i} r^{i-1}b(n-i,r)$$
                for $n\ge 2$; with the initial condition $b(1,r)=1$.
            \end{thm}

            % \begin{exmp}[$r=2$]
            %     The ordered $2$-Bell numbers are 
            %     \begin{align*}
            %         b(n,2)&=\sum_{N_1,N_2} S(n, N_1+2N_2)(N_1+2N_2)!\frac{(N_1+N_2)!}{N_1!N_2!}2^{-N_2}\\
            %         &= \sum_{N_1=0}^n \frac{1}{N_1!} \sum_{N_2=0}^{\lfloor \frac{n-N_1}{2} \rfloor} S(n, N_1+2N_2)(N_1+2N_2)!\frac{(N_1+N_2)!}{N_2!}2^{-N_2}.
            %     \end{align*}
            % \end{exmp}

    \section*{References}
        \printbibliography[title={\ }, heading=none]

\end{document}

%% file: w2n-LM-caterpillar-tree.tex
\begin{tikzpicture}[scale=2]
    \node [circle, draw, fill, inner sep =1pt, label = { 180: 1 }] (a0) at (0,0) {};
    \node [circle, draw, fill, inner sep =1pt] (a1) at (1,0) {};
    \node [circle, draw, fill, inner sep =1pt] (a11) at (0.7,0.5){};
    \node [circle, draw, fill, inner sep =1pt] (a12) at (0.9,0.5) {};
    \node [circle, draw, fill, inner sep =1pt] (a13) at (1.3,0.5) {};
    
    \node [circle, draw, fill, inner sep =1pt] (a2) at (2,0) {};
    \node [circle, draw, fill, inner sep =1pt] (a21) at (1.7,0.5){};
    \node [circle, draw, fill, inner sep =1pt] (a22) at (1.9,0.5) {};
    \node [circle, draw, fill, inner sep =1pt] (a23) at (2.3,0.5) {};
    
    \node [circle, draw, fill, inner sep =1pt] (a3) at (3,0) {};
    \node [circle, draw, fill, inner sep =1pt] (a31) at (2.7,0.5){};
    \node [circle, draw, fill, inner sep =1pt] (a32) at (2.9,0.5) {};
    \node [circle, draw, fill, inner sep =1pt] (a33) at (3.3,0.5) {};
    
    \node [circle, draw, fill, inner sep =1pt, label = { 0: 2 }] (a4) at (4,0) {};

    \draw (a0) -- (a1);
    \draw (a1) -- (a11);
    \draw (a1) -- (a12);
    \draw (a1) -- (a13);
    \node (etc) at ($ (a12)!0.5!(a13) $) {$\ldots$};
    \draw [decorate,decoration={brace,amplitude=5pt, raise=1pt},xshift=0pt,yshift=10pt]
    (a11.north) -- (a13.north) node [black,midway,xshift=0cm,yshift=0.5cm] 
    {$A_1$};
    
    \draw (a1) -- (a2);
    \draw (a2) -- (a21);
    \draw (a2) -- (a22);
    \draw (a2) -- (a23);
    \node (etc) at ($ (a22)!0.5!(a23) $) {$\ldots$};
    \draw [decorate,decoration={brace,amplitude=5pt, raise=1pt},xshift=0pt,yshift=10pt]
    (a21.north) -- (a23.north) node [black,midway,xshift=0cm,yshift=0.5cm] 
    {$A_2$};
    
    \node (etc) at ($ (a2)!0.5!(a3) $) {$\ldots$};
    \draw (a2) -- (2.25,0);
    \draw (a3) -- (2.75,0);
    
    \draw (a3) -- (a4);
    \draw (a3) -- (a31);
    \draw (a3) -- (a32);
    \draw (a3) -- (a33);
    \node (etc) at ($ (a32)!0.5!(a33) $) {$\ldots$};
    \draw [decorate,decoration={brace,amplitude=5pt, raise=1pt},xshift=0pt,yshift=10pt]
    (a31.north) -- (a33.north) node [black,midway,xshift=0cm,yshift=0.5cm] 
    {$A_k$};

\end{tikzpicture}

%% file: proof-of-thm-LM-ppoly-deg2-recursion.tex
For ease of notation, we write $\DMumford{H,L}$ for disjoint finite sets $H$ and $L$, to denote the heavy-light Hassett space $\DMumford{w_{|H|,|L|}}$, where $H$ and $L$ are indices of the heavy and light marked points respectively. Additionally, we use $*$ to denote the index of another heavy marked point.

We construct a map 
$$\psi: \LM{n+1} \to \left(\coprod_{L\subsetneq\{3,\ldots, n\}; L\neq \varnothing}(\DMumford{\{1,*\},L}\times \DMumford{\{2,*\}, \{3,\ldots, n\}\setminus L })\right)\coprod \DMumford{\{1,2\}, \{3,\ldots,n\}}$$
as follows.

Let $(C;x_{\bullet})$ be a $w_{2,n+1}$-stable pointed curve. Let $C^{\sing}\subseteq C$ be the set of nodes of $C$. That is, the heavy marked points are placed at $x_1, x_2\in C\setminus C^{\sing}$, and the light marked points are placed at $x_3, x_4, \ldots, x_{n+1}\in C\setminus C^{\sing}$. Let $C_0\cong \Pj^1$ be the (unique) smooth component of $C$ containing $x_{n+1}$. We consider three cases.\\

\noindent\textbf{Case 1:} $C_0$ has exactly two marked points and $|C\cap C^{\sing}|=1$. Define $\psi(C;x_{\bullet})$ to be the $n$-pointed curve obtained from $C$ by deleting $C_0\setminus C^{\sing}$; and replacing the unique marked point on $C_0\setminus \{x_{n+1}\}$ with the unique node (which becomes smooth after deletion) on $C_0$. \\

\noindent\textbf{Case 2:}
$C_0$ has at least four special points (marked points and nodes) and $x_2\notin C_0$. Let $q\in C^{\sing}\cap C_0$ be the unique node with the property that $C\setminus \{q\}$ has two connected components $C_1,C_2$ where $x_1,x_{n+1}\in C_1$ and $x_2\in C_2$. Define $\psi(C;x_{\bullet})$ to be the pair of marked curves $(\overline{C_1},\overline{C_2})$, where each $\overline{C_i}$ contains the marked points that were previously on $C_i\subseteq C$, together with an additional marked point $x_*=q\in \overline{C_i}$. \\

\noindent\textbf{Case 3:} $C_0$ has only one marked point and $|C\cap C^{\sing}|=2$. Let $q'\in C^{\sing}\cap C_0$ be the unique node with the property that $x_2$ and $x_{n+1}$ lie on different connected components of $C\setminus \{q'\}$. Let $C'_0\neq C_0$ be the other smooth component of $C$ that contains $q'$. Let $q\in C^{\sing}\cap C'_0$ be the unique node such that $q'$ and $x_2$ lie on different connected components of $C\setminus \{q\}$. We repeat our definition of $\psi(C;x_{\bullet})$ from Case 2 with this choice of $q$.
\\

\noindent\textbf{Case 4:}
$C_0$ has at least four special points and $x_2\in C_0$. Define $\psi(C;x_{\bullet})$ to be the curve $C$ itself, with the $n+1$-th marked point deleted. 
\\

It follows that for each element $D$ in the codomain of $\psi$, its fiber, $\inv{\psi}(D)$, is isomorphic to $\Pj^1\setminus \{pt\}\cong \A^1$, if $D$ is a pair of marked curves; and isomorphic to $\Pj^1$, if $D$ is a single $n$-pointed curve. Since $p_{\Pj^1}(x)=1+x$ and $p_{\A^1}(x)=x$, the desired recurrence readily follows.